\documentclass[11pt,letterpaper]{article}
\usepackage{amsthm}
\usepackage{amsmath}
\usepackage{amsfonts}
\usepackage{latexsym}
\usepackage{geometry}
\usepackage{enumerate}
\usepackage{csquotes}

\emergencystretch=\maxdimen
\title{A compactness theorem for four-dimensional shrinking gradient Ricci solitons}
\author{Yongjia Zhang}
\numberwithin{equation}{section}
\begin{document}
\maketitle

In \cite{haslhofer2011compactness} and \cite{haslhofer2015note}, Haslhofer and M{\"u}ller proved a compactness theorem for four-dimensional shrinking gradient Ricci solitons, with the only assumption being that the entropy is uniformly bounded from below. However, the limit in their result could possibly be an orbifold Ricci shrinker. In this paper we prove a compactness theorem for noncompact four-dimensional shrinking gradient Ricci solitons with a topological restriction and a noncollapsing assumption, that is, we consider Ricci shrinkers that can be embedded in a closed four-manifold with vanishing second homology group and are strongly $\kappa$-noncollapsed with respect to a universal $\kappa$. In particular, we do not need any curvature assumption and the limit is still a smooth nonflat shrinking gradient Ricci soliton.

\tableofcontents

\section{Introduction}

\newtheorem{Definition_1}{Definition}[section]
\newtheorem{Definition_2}[Definition_1]{Definition}
\newtheorem{Main_Theorem}[Definition_1]{Theorem}
\newtheorem{Main_Corollary}[Definition_1]{Corollary}

A triple $(M^n,g,f)$, where $(M^n,g)$ is a Riemannian manifold and $f$ is a function on $M^n$, is called a shrinking gradient Ricci soliton, or Ricci shrinker for short, if
\begin{eqnarray*}
Ric+\nabla^2 f=\frac{\lambda}{2}g,
\end{eqnarray*}
where $\lambda>0$ is a constant. By scaling the metric and adding a constant to the potential function $f$, the Ricci shrinker can always be normalized in the following way:
\begin{eqnarray}\label{eq:normalization}
Ric+\nabla^2f&=&\frac{1}{2}g,
\\\nonumber
|\nabla f|^2+R&=&f.
\end{eqnarray}
Shrinking gradient Ricci solitons are of great interest in the study of the singularity formation for the Ricci flow. For instance, they arise as fixed-point blow-up limits of Type I Ricci flows \cite{enders2011type}, and they are also blow-down limits of ancient solutions with nonnegative curvature operator \cite{perelman2002entropy}. For other works on the Ricci shrinkers, please refer to \cite{li2016four}, \cite{munteanu2014conical}, \cite{munteanu2015geometry}, \cite{munteanu2015positively}, \cite{munteanu2016structure}, and \cite{naber2010noncompact}. In this paper, we restrict our attention to the shrinking gradient Ricci solitons that can be embedded in a closed four-manifold with vanishing second homology group. This condition is also considered by Bamler and Zhang \cite{bamler2015heat}. Besides that, we impose a uniform strong noncollapsing assumption, which fortunately holds for singularity models; see below. We define the following space of Ricci shrinkers.
\\

\begin{Definition_1} \label{definition}
Given $\kappa>0$, $\mathcal{M}^4(\kappa)$ is the collection of all the four-dimensional noncompact shrinking gradient Ricci solitons $(M^4,g,f,p)$, where $p$ is the point at which $f$ attains its minimum, satisfying
\begin{enumerate}[(a)]
  \item $(M^4,g)$ is nonflat;
  \item $M^4$ can be embedded in a closed four-manifold $N^4$ with $H_2(N)=0$, where $H_2$ is the second homology group with coefficients in $\mathbb{Z}$;
  \item $(M^4,g)$ is strongly $\kappa$-noncollapsed on all scales.
\end{enumerate}
\end{Definition_1}

\bigskip

Notice that in item (b) above the closed manifold $N^4$ may depend on $M$; we do not assume that every Ricci shrinker in $\mathcal{M}^4(\kappa)$ can be embedded in the same closed four-manifold. By strong noncollapsing we mean the following.
\\

\begin{Definition_2}
A Riemannian manifold $(M^n,g)$ is strongly $\kappa$-noncollapsed on all scales, where $\kappa>0$, if the following holds. For all $x\in M$ and $r>0$, if $R<r^{-2}$ on $B(x,r)$, then $\operatorname{Vol}(B(x,r))\geq\kappa r^n$. Here we use $R$ to denote the scalar curvature.
\end{Definition_2}

\bigskip

Our main theorem is the following.
\\

\begin{Main_Theorem} \label{Main_Theorem}
$\mathcal{M}^4(\kappa)$ is compact in the smooth pointed Cheeger-Gromov sense, where each $(M^4,g,f,p)\in\mathcal{M}^4(\kappa)$ is normalized as in (\ref{eq:normalization}).
\end{Main_Theorem}

\bigskip

Here by compact we mean that for every sequence $\{(M^4_k,g_k,f_k,p_k)\}_{k=1}^\infty$ that is contained in $\mathcal{M}^4(\kappa)$, there exists a subsequence that converges in the pointed smooth Cheeger-Gromov sense to a Ricci shrinker $(M^4_\infty,g_\infty,f_\infty,p_\infty)\in \mathcal{M}^4(\kappa)$.
\\

A homotopy four-sphere, as a particular example, has vanishing second homology group. When approaching the four-dimensional smooth Poincar\'{e} conjecture using the Ricci flow, the Ricci shrinkers that can be embedded in a homotopy four-sphere are the only ones that may arise when analyzing the singularities at the first singular time, and are the only ones that can be embedded in the smooth part of the post-surgery objects\footnote{Ricci flow with surgery on closed four-manifolds may possibly produce orbifolds. For instance, the surgery on $\mathbb{S}^4$ at a pinching neck over $\mathbb{S}^3/\mathbb{Q}_8$ produces two orbifolds, where $\mathbb{Q}_8$ stands for the quaternion group. We would like to thank Professor Bennett Chow for pointing this out.}. The strong noncollapsing assumption is also motivated by the finite-time singularity analysis for the Ricci flow. Suppose $(M,g(t))_{t\in[0,T)}$ is a Ricci flow on a closed manifold which develops a singularity at time $T$, then according to Perelman \cite{perelman2002entropy}, $(M,g(t))_{t\in[0,T)}$ is strongly $\kappa$-noncollapsed on some scale $r_0$, where both $\kappa$ and $r_0$ depend only on $g(0)$ and $T$ (see Theorem 6.74 in \cite{chow2007ricci}). Thus any blow-up limit of $(M,g(t))_{t\in[0,T)}$ must be strongly $\kappa$-noncollapsed on all scales. We hope our result will be helpful to the finite-time singularity analysis for the Ricci flow on four-dimensional closed manifolds. As a complement, we would like to mention that, as pointed out by Bennett Chow, it is interesting to ask whether or not the curvature of every four-dimensional nontrivial noncollapsed steady gradient Ricci soliton has some sense of positivity, the validity of which should facilitate the study of Type II singularities in the sense of Hamilton \cite{hamilton1995formation}.
\\

The condition of being embedded in a closed four-manifold with vanishing second homology group plays a very important role in ruling out the Ricci-flat limits. By the strong noncollapsing property, a Ricci-flat blow-up limit of a Ricci flow at a finite-time singularity must have Euclidean volume growth, and, according to Cheeger and Naber \cite{cheeger2015regularity}, must be \textit{asymptotically locally Euclidean} (ALE for short), which cannot be embedded in any closed four-manifold with vanishing second homology group (see Corollary 5.8 in Anderson \cite{anderson2010survey}; an alternative proof by Richard Bamler is provided in section 6). This idea gives a uniform curvature growth estimate for every element in the space $\mathcal{M}^4(\kappa)$; see Theorem \ref{BCBD} below, from which we obtain the compactness result. This argument is in the spirit of Perelman's bounded curvature at bounded distance result for $\kappa$-solutions with nonnegative curvature operator (see section 11 of \cite{perelman2002entropy}), where Perelamn also assumes a uniform $\kappa$, which is motivated by the same reason that all these $\kappa$-solutions arise from the same Ricci flow that forms a finite-time singularity. However, there is always a universal $\kappa$ for all the three-dimensional $\kappa$-solutions that is not a shrinking space form because of the classification of three-dimensional Ricci shrinkers.
\\

From the proof of the main theorem, we also get the following property of the space $\mathcal{M}^4(\kappa)$.
\\

\begin{Main_Corollary} \label{Main_Corollary}
There exist $C_1>0$, $C_2>0$, and $C_3<\infty$ depending only on $\kappa$, such that for every $(M^4,g,f,p)\in\mathcal{M}^4(\kappa)$ the following hold.
\begin{enumerate}[(a)]
  \item $R(p)>C_1$.
  \item $R(x)>C_2f^{-1}(x)$, for all $x\in M$.
  \item $|\pi_1(M)|<C_3$, where $\pi_1(M)$ is the fundamental group of $M$.
\end{enumerate}
\end{Main_Corollary}

\bigskip

It is worth mentioning that in Haslhofer and M{\"u}ller \cite{haslhofer2011compactness} \cite{haslhofer2015note}, they also proved a compactness theorem for four-dimensional Ricci shrinkers, where they only assume a uniform lower bound of the entropy, but the limit could possibly be an orbifold shrinker. In comparison with their result, the strong noncollapsing assumption in our theorem is correspondent to their bounded entropy assumption (indeed, it is clear that a uniform lower bound of entropy implies $\kappa$-noncollapsing with respect to a universal $\kappa$; see Carrillo and Ni \cite{carrillo2009sharp} and Yokota \cite{yokota2012addendum}, yet we do not know how it is related to our strong noncollapsing assumption); in addition, we have a topological restriction. What is new in our work is that the orbifold Ricci shrinkers will never show up as the limits.
\\

This paper is organized as follows. In section 2 we collect some known results for the Ricci shrinkers, which are used in our arguments. In section 3 we carry out some a priori estimates. In section 4 we prove a bounded curvature at bounded distance theorem for the Ricci shrinkers in $\mathcal{M}^4(\kappa)$. In section 5 we prove Theorem \ref{Main_Theorem} and Corollary \ref{Main_Corollary}. In section 6 we provide an alternative proof of Anderson's theorem \cite{anderson2010survey}.
\\

\section{Preliminaries}
\newtheorem{Curvature_Evolution}{Proposition}[section]
\newtheorem{Potential_Growth}[Curvature_Evolution]{Proposition}
\newtheorem{Volume_Growth}[Curvature_Evolution]{Proposition}
\newtheorem{Sobolev_Inequality}[Curvature_Evolution]{Proposition}
\newtheorem{Gap_Theorem}[Curvature_Evolution]{Proposition}

In this section we collect some well-known results that are used in our proof. Notice that in this paper the Ricci shrinker equation that we work with may not be normalized as (\ref{eq:normalization}), since sometimes scaling is necessary. Hence we will specify the Ricci shrinker equation in every statement. We start with the following differential equations for the geometric quantities on the Ricci shrinkers.
\\

\begin{Curvature_Evolution}
Let $(M,g,f)$ be a shrinking gradient Ricci soliton satisfying
\begin{eqnarray*}
Ric+\nabla^2f=\frac{\lambda}{2}g,
\end{eqnarray*}
where $\lambda>0$. Then the following hold.
\begin{eqnarray}
\Delta_f R &=& \lambda R-2|Ric|^2, \label{eq:R_evolve}
\\
\Delta_f Ric &=& \lambda Ric + Rm*Ric, \label{eq:Ric_evolve}
\\
\Delta_f Rm &=& \lambda Rm + Rm*Rm, \label{eq:Rm_evolve}
\\
\Delta_f \nabla^k Rm &=& (\frac{k}{2}+1)\lambda \nabla^k Rm + \sum_{j=0}^k\nabla^jRm*\nabla^{k-j}Rm, \label{eq:CoRm_evolve}
\end{eqnarray}
where $*$ stands for some contraction and $\Delta_f=\Delta-\langle\nabla f,\nabla\cdot\ \rangle$ is the $f$-Laplacian operator.
\end{Curvature_Evolution}

\begin{proof}
Since every reference on these differential equations we can find deals only with the case $\lambda=1$, we take (\ref{eq:Rm_evolve}) as an example to quickly sketch how these formulae can be carried out; other equations can be proved in the same way. Recall that the canonical form of a Ricci shrinker $g(t)=\tau(t)\varphi_t^*(g)$ evolves by the Ricci flow (see Theorem 4.1 of \cite{chow2006hamilton}), where
\begin{eqnarray*}
\tau(t)&=&1-\lambda t,
\\
\frac{d}{dt}\varphi_t&=&\frac{1}{\tau}\nabla f\circ\varphi_t,
\\
\varphi_0&=&id.
\end{eqnarray*}

If we regard $Rm$ as a $(4,0)$ tensor, then we have $Rm(g(t))=\tau(t) Rm(\varphi_t^*(g))=\tau(t)\varphi_t^*(Rm)$. Hence by the standard curvature evolution equation (see Theorem 7.1 of \cite{hamilton1982three}) we have
\begin{eqnarray*}
\Delta Rm+Rm*Rm=\left.\frac{\partial}{\partial t}\right\vert_{t=0}Rm(g(t))=-\lambda Rm + \mathcal{L}_{\nabla f}Rm,
\end{eqnarray*}
where $\mathcal{L}$ stands for the Lie derivative. Let $Y_1$, $Y_2$, $Y_3$, $Y_4$ be four arbitrary vector fields. Then we have
\begin{eqnarray*}
\mathcal{L}_{\nabla f}Rm(Y_1,Y_2,Y_3,Y_4)&=&\nabla_{\nabla f}(Rm(Y_1,Y_2,Y_3,Y_4))-Rm(\mathcal{L}_{\nabla f}Y_1,Y_2,Y_3,Y_4)
\\
&&-Rm(Y_1,\mathcal{L}_{\nabla f}Y_2,Y_3,Y_4)-Rm(Y_1,Y_2,\mathcal{L}_{\nabla f}Y_3,Y_4)
\\
&&-Rm(Y_1,Y_2,Y_3,\mathcal{L}_{\nabla f}Y_4)
\\
&=&\nabla_{\nabla f}Rm(Y_1,Y_2,Y_3,Y_4)+Rm(\nabla_{Y_1}\nabla f,Y_2,Y_3,Y_4)
\\
&&+Rm(Y_1,\nabla_{Y_2}\nabla f,Y_3,Y_4)+Rm(Y_1,Y_2,\nabla_{Y_3}\nabla f,Y_4)
\\
&&+Rm(Y_1,Y_2,Y_3,\nabla_{Y_4}\nabla f).
\end{eqnarray*}
Taking into account that $\displaystyle\nabla^2 f=\frac{\lambda}{2}g-Ric$ we obtain the conclusion.
\end{proof}

\bigskip

The following two propositions for the potential function growth rate and the volume growth rate were proved by Cao and Zhou \cite{cao2010complete}, Munteanu \cite{munteanu2009volume}. We use its sharpened version, whose proof is due to Haslhofer and M\"{u}ller \cite{haslhofer2011compactness}. Besides that, Munteanu and Wang \cite{munteanu2014geometry} proved a volume growth estimate with a better constant.
\\

\begin{Potential_Growth}  \label{Potential_Growth}
Let $(M^n,g,f)$ be a noncompact shrinking gradient Ricci soliton normalized as in (\ref{eq:normalization}). Let $p$ be a point where $f$ attains its minimum. Then the following holds.
\begin{eqnarray} \label{eq:potential}
\frac{1}{4}(d(x,p)-5n)_+^2\leq f(x)\leq \frac{1}{4} (d(x,p)+\sqrt{2n})^2,
\end{eqnarray}
where $u_+:=\max\{u,0\}$ denotes the positive part of a function.
\end{Potential_Growth}

\bigskip

\begin{Volume_Growth}
There exists $C<\infty$ depending only on the dimension $n$, such that under the same assumption of Proposition \ref{Potential_Growth} the following holds.
\begin{eqnarray} \label{eq:volume}
\operatorname{Vol}(B(p,r))\leq Cr^n,
\end{eqnarray}
for all $r>0$.
\end{Volume_Growth}

\bigskip

To locally estimate the Ricci curvature, we need the following local Sobolev inequality, whose constant depends only on the local geometry.
\\

\begin{Sobolev_Inequality} \label{prop_Sobolev_Inequality}
For all $\kappa>0$, there exists $C<\infty$ and $\delta\in(0,2)$, depending only on $\kappa$ and the dimension $n\geq 3$ such that the following holds. Let $(M^n,g)$ be a Riemannian manifold and $x_0\in M$, and assume that $|Rm|\leq2$ on $B(x_0,2)$ and $\operatorname{Vol}(B(x_0,2))\geq\kappa$. Then
\begin{eqnarray} \label{eq:Sobolev}
\|u\|_{L^{\frac{2n}{n-2}}}\leq C \|\nabla u\|_{L^2},
\end{eqnarray}
for all $u\in C_{0}^{\infty}(B(x_0,\delta))$.
\end{Sobolev_Inequality}

\begin{proof}
This is a standard result; for the convenience of the readers we sketch the proof. We follow the lines of reasoning of Lemma 3.2 of Haslhofer and M{\"u}ller \cite{haslhofer2011compactness}. We need only to prove an $L^1$ Sobolev inequality
\begin{eqnarray} \label{eq:Sobolev_L1}
\|u\|_{L^{\frac{n}{n-1}}}\leq C_1 \|\nabla u\|_{L^1},
\end{eqnarray}
for all $u\in C_{0}^{\infty}(B(x_0,\delta))$, where $\delta$ and $C_1$ depend only on $\kappa$ and the dimension $n$. Then (\ref{eq:Sobolev}) follows from (\ref{eq:Sobolev_L1}). Indeed, $C_1$ is equal to the isoperimetric constant of $B(x_0,\delta)$
\begin{eqnarray*}
C_1=C_I=\sup |\Omega|^{\frac{n}{n-1}}/|\partial\Omega|,
\end{eqnarray*}
where the supremum is taken over all the open sets $\Omega\subset\subset B(x_0,\delta)$ with smooth boundary. By a theorem of Croke (Theorem 11 in \cite{croke1980some}), $C_I$ can be estimated by
\begin{eqnarray*}
C_I\leq C(n)\omega^{-\frac{n+1}{n}},
\end{eqnarray*}
where $C(n)$ is a constant depending only on the dimension and $\omega$ is the visibility angle defined by
\begin{eqnarray*}
\omega=\inf_{y\in B(x_0,\delta)}|U_y|/|\mathbb{S}^{n-1}|,
\end{eqnarray*}
where $U_y=\{v\in T_y B(x_0,\delta): |v|=1, \text{the geodesic }\gamma_v\text{ minimizes up to }\partial B(x_0,\delta)\}$.
\\

We restrict $\delta$ in $(0,\frac{1}{2})$ and let $y$ be an arbitrary point in $B(x_0,\delta)$. Let
\begin{eqnarray*}
J(r,\theta)dr\wedge d\theta,\ \bar{J}(r,\theta)dr\wedge d\theta
\end{eqnarray*}
be the volume elements in terms of the polar normal coordinates around the point $y$ and in the hyperbolic space with constant sectional curvature $-2$, respectively. By the relative volume comparison theorem, we have
\begin{eqnarray*}
c_2\kappa-C_3\delta^n&\leq&|B(x_0,1)|-|B(x_0,\delta)|\leq\int_{U_y}\int_0^{1+\delta}J(r,\theta)drd\theta
\\
&\leq&\int_{U_y}\int_0^{1+\delta}\bar{J}(r,\theta)drd\theta\leq C_4|U_y|\left(\frac{3}{2}\right)^n,
\end{eqnarray*}
where $c_2$, $C_3$, and $C_4$ are constants depending only on the dimension $n$. Taking $\displaystyle \delta=\left(\frac{c_2\kappa}{2C_3}\right)^{\frac{1}{n}}$, we have that $|U_y|$ is bounded from below by a constant depending only on $\kappa$ and the dimension $n$, for all $y\in B(x_0,\delta)$, whence follows the conclusion.

\end{proof}

\bigskip

We conclude this section with the following gap theorem of Yokota \cite{yokota2009perelman} \cite{yokota2012addendum}, which is of great importance in proving that the limit shrinker is nonflat.
\\

\begin{Gap_Theorem} \label{Gap}
There exists $\varepsilon>0$ depending only on the dimension $n$ such that the following holds. Let $(M^n,g,f)$ be a shrinking gradient Ricci soliton, which is normalized as in (\ref{eq:normalization}). If
\begin{eqnarray*}
\frac{1}{(4\pi)^{\frac{n}{2}}}\int_M e^{-f}dg>1-\varepsilon,
\end{eqnarray*}
then $(M^n,g,f)$ is the Gaussian shrinker, that is, $(M^n,g)$ is the Euclidian space.
\end{Gap_Theorem}
\bigskip

\section{A priori estimates}

\newtheorem{Shi}{Proposition}[section]
\newtheorem{Ricci_Smallness}[Shi]{Proposition}

The a priori estimates in this section hold for any dimension $n\geq 3$. We start with a localized derivative estimate for the Riemann curvature tensor.
\\

\begin{Shi}\label{Shi}
There exists $C<\infty$ depending only on the dimension $n$ such that the following holds. Let $(M^n,g,f)$ be a shrinking gradient Ricci soliton such that
\begin{eqnarray*}
Ric+\nabla^2 f=\frac{\lambda}{2}g,
\end{eqnarray*}
where $\lambda>0$. Let $x_0\in M$ and $r>0$. If $|Rm|\leq r^{-2}$ and $|\nabla f|\leq r^{-1}$ on $B(x_0,2r)$, then
\begin{eqnarray} \label{eq:1st_deriv}
|\nabla Rm|\leq Cr^{-3}
\end{eqnarray}
on $B(x_0,r)$. More generally, there exist $C_l$ depending only on $l\geq 0$ and the dimension $n$, such that under the above assumptions, it holds that
\begin{eqnarray} \label{eq:2st_deriv}
|\nabla^l Rm|\leq C_l r^{-2-l}
\end{eqnarray}
on $B(x_0,r)$.
\end{Shi}

\begin{proof}
The proof is a standard elliptic modification of Shi's estimates \cite{shi1989deforming}; we will only show (\ref{eq:1st_deriv}); the higher derivative estimates (\ref{eq:2st_deriv}) follow in a standard way by induction. We compute using (\ref{eq:Rm_evolve})
\begin{eqnarray*}
\Delta_f|Rm|^2&=&2\langle Rm, \Delta_f Rm\rangle+2|\nabla Rm|^2
\\
&=&2|\nabla Rm|^2+2\lambda |Rm|^2+Rm*Rm*Rm
\\
&\geq&2|\nabla Rm|^2-C_1 |Rm|^3,
\end{eqnarray*}
where $C_1<\infty$ depends only on the dimension $n$. By (\ref{eq:CoRm_evolve}), we have
\begin{eqnarray*}
\Delta_f|\nabla Rm|^2&=&2\langle\nabla Rm, \Delta_f \nabla Rm\rangle+2|\nabla^2 Rm|^2
\\
&=&2|\nabla^2Rm|^2+3\lambda |\nabla Rm|^2 +Rm*\nabla Rm*\nabla Rm
\\
&\geq&2|\nabla^2Rm|^2-C_2|Rm||\nabla Rm|^2,
\end{eqnarray*}
where $C_2<\infty$ depends only on the dimension $n$.
\\

Define $u=(\beta r^{-4}+|Rm|^2)|\nabla Rm|^2$, where $\beta>0$ is a constant that we will specify later, we have
\begin{eqnarray*}
\Delta_f u &=& |\nabla Rm|^2\Delta_f|Rm|^2+(\beta r^{-4}+|Rm|^2)\Delta_f|\nabla Rm|^2+2\langle \nabla |Rm|^2,\nabla|\nabla Rm|^2\rangle
\\
&\geq&2|\nabla Rm|^4-C_1|\nabla Rm|^2|Rm|^3+(\beta r^{-4}+|Rm|^2)(2|\nabla^2Rm|^2-C_2|Rm||\nabla Rm|^2)
\\
&&-8|\nabla Rm|\cdot|\nabla|Rm||\cdot|\nabla^2Rm|\cdot|Rm|
\\
&\geq&2|\nabla Rm|^4-C_1|\nabla Rm|^2|Rm|^3+(\beta r^{-4}+|Rm|^2)(2|\nabla^2Rm|^2-C_2|Rm||\nabla Rm|^2)
\\
&&-\frac{1}{2}|\nabla Rm|^4-32|\nabla^2Rm|^2|Rm|^2,
\end{eqnarray*}
where we have used Kato's inequality as well as the Cauchy-Schwarz inequality. Letting $\beta=16$ and taking into account that $|Rm|\leq r^{-2}$ in $B(x_0,2r)$, we have
\begin{eqnarray*}
\Delta_f u&\geq& \frac{3}{2}|\nabla Rm|^4-C_3r^{-6}|\nabla Rm|^2
\\
&\geq&|\nabla Rm|^4-C_4r^{-12},
\end{eqnarray*}
where we have used the Cauchy-Schwarz inequality, and $C_3$ and $C_4$ are constants depending only on $n$. By the definition of $u$ we have $\displaystyle|\nabla Rm|^4\geq\frac{r^8}{289}u^2$; hence we have
\begin{eqnarray}
\Delta_f u\geq c_5r^8u^2-C_5r^{-12},
\end{eqnarray}
where $c_5$ and $C_5$ are constants depending only on $n$.
\\

We let $\phi(x)=\varphi(d(x_0,x))$ be the cut-off function, where $\varphi(s)=0$ for $s\geq 2r$, $\varphi(s)=1$ for $s\in[0,r]$, and
\begin{eqnarray} \label{eq:cut_off}
0\leq\varphi&\leq&1,
\\\nonumber
-2r^{-1}\leq\varphi'(s)&\leq&0,
\\\nonumber
|\varphi''(s)|&\leq& 2 r^{-2},
\end{eqnarray}
for all $s\in[r,2r]$. We compute
\begin{eqnarray} \label{eq:middle_step}
\Delta_f(u\phi^2)&=&\phi^2\Delta_f u+u\Delta_f \phi^2+2\langle\nabla u, \nabla\phi^2\rangle
\\\nonumber
&\geq& c_5r^8u^2\phi^2-C_5r^{-12}\phi^2+2\langle\nabla(u\phi^2),\nabla\log{\phi^2}\rangle-8|\nabla\phi|^2u+u\Delta_f\phi^2.
\end{eqnarray}
The last two terms in (\ref{eq:middle_step}) need to be estimated. We have
\begin{eqnarray*}
|\nabla \phi|^2=\varphi'^2|\nabla d|^2\leq4r^{-2},
\end{eqnarray*}
and
\begin{eqnarray*}
\Delta_f\phi^2&=&2\phi(\varphi'\Delta_fd+\varphi''|\nabla d|^2)+2\varphi'^2|\nabla d|^2
\\
&=&2\phi(\varphi'\Delta d-\varphi'\langle\nabla f,\nabla d\rangle)+2\phi\varphi''|\nabla d|^2+2\varphi'^2|\nabla d|^2
\\
&\geq&2\left(-\frac{2(n-1)\coth{(1)}}{r^2}-\frac{2}{r^2}\right)-\frac{4}{r^2}\geq-C_6r^{-2},
\end{eqnarray*}
where $C_6$ is a positive constant depending only on the dimension $n$. In the above derivation we have used $|\nabla f|\leq r^{-1}$, the properties of $\varphi$ (\ref{eq:cut_off}), the Laplacian comparison theorem, and that $\varphi'(s)=0$ for all $s\in[0,r]$. Inserting the above inequalities into (\ref{eq:middle_step}), and defining $G=u\phi^2$, we have
\begin{eqnarray} \label{eq:middle_step_2}
\Delta_fG\geq c_5r^8\frac{G^2}{\phi^2}-C_5 r^{-12}\phi^2+2\langle\nabla G,\nabla\log{\phi^2}\rangle-C_7r^{-2}\frac{G}{\phi^2}.
\end{eqnarray}
Let $x_1\in B(x_0,2r)$ be a point where $G$ attains its maximum. Taking into account that $0\leq\phi\leq 1$, it follows from (\ref{eq:middle_step_2}) that
\begin{eqnarray*}
c_5r^8 G(x_1)^2-C_7 r^{-2}G(x_1)-C_5r^{-12}\leq 0,
\end{eqnarray*}
which solves $G(x_1)\leq C_8r^{-10}$, where $C_7$ and $C_8$ depend only on $n$. Therefore $u(x)\leq C_8r^{-10}$ on $B(x_0,r)$, where $\phi=1$ and $G=u$. It follows from the definition of the function $u$ that
\begin{eqnarray*}
|\nabla Rm|^2\leq C r^{-6}
\end{eqnarray*}
on $B(x_0,r)$.
\end{proof}

\bigskip

The following proposition says the smallness of the scalar curvature on a ball implies the smallness of the Ricci curvature on a smaller ball. Our argument is inspired by Theorem 3.2 in \cite{wang2012conditions}. Bamler and Zhang have implemented the same idea in \cite{bamler2015heat}.
\\

\begin{Ricci_Smallness} \label{Ricci_Smallness}
For any $\kappa>0$, there exists $\delta\in(0,2)$ and $C<\infty$, depending only on $\kappa$ and the dimension $n$, such that the following holds. Let $(M^n,g,f)$ be a shrinking gradient Ricci soliton such that
\begin{eqnarray*}
Ric+\nabla^2 f=\frac{\lambda}{2}g,
\end{eqnarray*}
where $\lambda>0$. Let $x_0\in M$ and $r\in(0,1]$. If
\begin{eqnarray*}
|Rm|\leq 2,\ R\leq r^2,\ \text{and}\ |\nabla f|\leq r
\end{eqnarray*}
on $B(x_0,2)$, and
\begin{eqnarray*}
\operatorname{Vol}(B(x_0,2))\geq\kappa.
\end{eqnarray*}
Then
\begin{eqnarray} \label{eq:Ricci_small}
|Ric|\leq Cr
\end{eqnarray}
on $\displaystyle B(x_0,\frac{\delta}{2})$.
\end{Ricci_Smallness}

\begin{proof}
We define a cut-off function that is similar to the one that we have used in the proof of the last proposition. Let $\phi(x)=\varphi(d(x_0,x))$, where $\varphi(s)=0$ for $s\geq 2$, $\varphi(s)=1$ for $s\in[0,1]$, and
\begin{eqnarray} \label{eq:cut_off_2}
0\leq\varphi&\leq&1,
\\\nonumber
-2\leq\varphi'(s)&\leq&0,
\\\nonumber
|\varphi''(s)|&\leq& 2,
\end{eqnarray}
for $s\in[1,2]$. Integrating the equation (\ref{eq:R_evolve}) against $\phi$, we have
\begin{eqnarray*}
2\int|Ric|^2\phi&=&\lambda\int R\phi - \int\phi\Delta R +\int\langle\nabla f,\nabla R\rangle\phi
\\
&=&\lambda\int R\phi-\int R\Delta\phi-\int\phi R\Delta f-\int R\langle\nabla\phi,\nabla f\rangle
\\
&=&(1-\frac{n}{2})\lambda\int R\phi-\int R\Delta\phi+\int \phi R^2-\int R\langle\nabla\phi,\nabla f\rangle
\\
&\leq&-\int R\Delta\phi+\int \phi R^2-\int R\langle\nabla\phi,\nabla f\rangle
\end{eqnarray*}
where we have used $\displaystyle\Delta f=\frac{n}{2}\lambda-R$ and Chen's result \cite{chen2009strong} that $R\geq 0$ on a Ricci shrinker. By the Laplacian comparison theorem, the Bishop-Gromov volume comparison theorem, and the property of $\phi$ (\ref{eq:cut_off_2}), we have
\begin{eqnarray*}
-\Delta\phi&\leq& C_1,
\\
\langle\nabla\phi,\nabla f\rangle&\leq& C_2 r,
\\
\operatorname{Vol}(B(x_0,2))&\leq&C_3,
\end{eqnarray*}
where $C_1$, $C_2$, and $C_3$ are positive constants depending only on the dimension $n$. It then follows that
\begin{eqnarray*}
\int|Ric|^2\phi\leq C_1C_3r^2+C_3 r^4+C_2C_3r^3,
\end{eqnarray*}
and that
\begin{eqnarray*}
\|Ric\|_{L^2(B(x_0,\delta))}\leq C_4 r,
\end{eqnarray*}
where $C_4$ depends only on the dimension $n$, and $\delta\in(0,2)$ is the positive number given by Proposition \ref{prop_Sobolev_Inequality} that depends only on $\kappa$ and the dimension $n$.
\\

We have the following inequality satisfied by $|Ric|$:
\begin{eqnarray*}
2|Ric|\Delta_f|Ric|+2|\nabla|Ric||^2&=&\Delta_f|Ric|^2=2\langle Ric,\Delta_f Ric\rangle+2|\nabla Ric|^2
\\
&=&2\lambda|Ric|^2+2|\nabla Ric|^2-Rm*Ric*Ric
\\
&\geq&2|\nabla Ric|^2-Rm*Ric*Ric.
\end{eqnarray*}
Taking into account that $|Rm|\leq 2$ on $B(x_0,2)$ and Kato's inequality that $|\nabla|Ric||^2\leq|\nabla Ric|^2$, we have
\begin{eqnarray} \label{eq:ineq:Ricci}
\Delta_f|Ric|\geq -C_5 |Ric|,
\end{eqnarray}
where $C_5$ depends on the dimension $n$. We use the local Sobolev inequality (\ref{eq:Sobolev}) to apply the standard Moser iteration to the inequality (\ref{eq:ineq:Ricci}). Notice that we need to use $|\nabla f|\leq r\leq1$ when performing the iteration. Indeed, this is the only reason why we have to put a restriction on the scale $r$. It follows that
\begin{eqnarray*}
\sup_{B(x_0,\frac{\delta}{2})}|Ric|\leq C_6\|Ric\|_{L^2(B(x,\delta))}\leq Cr,
\end{eqnarray*}
where $C$ depends only on $\kappa$ and the dimension $n$.

\end{proof}
\bigskip

\section{Bounded curvature at bounded distance}

\newtheorem{BCBD}{Theorem}[section]

In this section we prove a bounded curvature at bounded distance theorem for Ricci shrinkers in the space $\mathcal{M}^4(\kappa)$. This result is an analogue to Perelman's bounded curvature at bounded distance result (see section 11 of \cite{perelman2002entropy}). The fact that the Ricci-flat limit does not appear in our argument plays an equally important role as the fact that the asymptotic volume ratio equals zero in Perelman's argument. However, our result is somewhat weaker than Perelman's. We are only able to fix the base point where the potential function attains its minimum (or wherever is at a bounded distance to it). The reason is because when a point is far off from the minimum point of the potential function, the largeness of the curvature scale is no longer characterized by $1$, but by $|\nabla f|^2$. Suppose around a point the curvature is large compared to $1$ but small compared to $|\nabla f|^2$, then the a priori estimates we have established in the previous section do not hold any more, since the assumption $|\nabla f|\leq r\leq 1$ made in Proposition \ref{Ricci_Smallness} is no longer valid after scaling, and Moser iteration does not yield a nice bound for the Ricci curvature as (\ref{eq:Ricci_small}).
\\

\begin{BCBD} \label{BCBD}
There exists $C<\infty$ and $D<\infty$ depending only on $\kappa$, such that the following holds. Let $(M^4,g,f,p)\in\mathcal{M}^4(\kappa)$ be normalized as in (\ref{eq:normalization}). Then it holds that
\begin{eqnarray*}
|Rm|(x)&\leq&C,\ \text{if}\ x\in B(p,200),
\\
\frac{|Rm|}{f}(x)&\leq& D,\ \text{if}\ x\notin B(p,200).
\end{eqnarray*}
\end{BCBD}

\begin{proof}
\newtheorem{Point_Picking}{Claim}

We argue by contradiction. Suppose the statement is not true, then there exist a sequence of counterexamples $\{(M^4_k,f_k,g_k,p_k)\}_{k=1}^\infty$ normalized as in (\ref{eq:normalization}), and $x_k\in M_k$, such that for all $k\geq 1$, either
\\

\noindent(a) $x_k\in B_{g_k}(p_k,200)$ and $|Rm_k|(x_k)\geq k$,
\\

\noindent or,
\\

\noindent(b) $x_k\notin B_{g_k}(p_k,200)$ and $\displaystyle\frac{|Rm_k|}{f_k}(x_k)\geq k$.
\\

Notice that by (\ref{eq:potential}), we have $f_k(x)\geq 1000$ whenever $x\notin B(p_k,200)$; hence $|Rm_k|(x_k)\rightarrow\infty$.
\\

The following standard point picking technique is due to Perelman \cite{perelman2002entropy}.
\\

\begin{Point_Picking} \label{Point_Picking}
There exists $A_k\rightarrow\infty$ and $y_k\in B_{g_k}(x_k,1)$, such that
\begin{eqnarray} \label{eq:Point_Picking}
|Rm_k|(x)\leq 2Q_k,\ \text{for all}\ x\in B_{g_k}(y_k,A_kQ_k^{-\frac{1}{2}})\subset B_{g_k}(x_k,2),
\end{eqnarray}
where $Q_k=|Rm_k|(y_k)\geq|Rm_k|(x_k)$.
\end{Point_Picking}

\begin{proof}[Proof of the Claim]
Denote $Q_k^{(0)}=|Rm_k|(x_k)$ and $y_k^{(0)}=x_k$, let $\displaystyle A_k=\frac{1}{100}\left(Q_k^{(0)}\right)^{\frac{1}{2}}\rightarrow\infty$. We start from $y_k^{(0)}$. Suppose $y_k^{(j)}$ is chosen, and cannot be taken as $y_k$. Let $|Rm_k|\left(y_k^{(j)}\right)=Q_k^{(j)}$, then there exists $\displaystyle y_k^{(j+1)}\in B_{g_k}\left(y_k^{(j)},A_k\left(Q_k^{(j)}\right)^{-\frac{1}{2}}\right)$, such that $Q_k^{(j+1)}=|Rm_k|\left(y_k^{(j+1)}\right)\geq 2Q_k^{(j)}$. Hence we have
\begin{eqnarray*}
dist_{g_k}\left(y_k^{(0)},y_k^{(j+1)}\right)&\leq& A_k\left(Q_k^{(0)}\right)^{-\frac{1}{2}}+ A_k\left(Q_k^{(1)}\right)^{-\frac{1}{2}}+...+ A_k\left(Q_k^{(j)}\right)^{-\frac{1}{2}}
\\
&\leq& A_k\left(Q_k^{(0)}\right)^{-\frac{1}{2}}\left(1+\frac{1}{\sqrt{2}}+...+\left(\frac{1}{\sqrt{2}}\right)^j+...\right)
\\
&\leq&\frac{1}{100}\times 4,
\end{eqnarray*}
and it follows that $y_k^{(j)}\in B_{g_k}(x_k,1)$ for all $j\geq 0$. This procedure must terminate in a finite steps since the manifold $M_k$ is smooth; then the last element chosen by this procedure can be taken as $y_k$.
\end{proof}

\bigskip

Since for any $k\geq 1$ there can be only two cases (a) or (b), then either for infinitely many $k$, (a) holds, or, for infinitely many $k$, (b) holds. By passing to a subsequence, we need only to deal with the following two cases.
\\

\noindent Case \textbf{I.} $x_k\in B_{g_k}(p_k,200)$ and $|Rm_k|(x_k)\geq k$, for all $k\geq 1$.
\\

\noindent Case \textbf{II.} $x_k\notin B_{g_k}(p_k,200)$ and $\displaystyle\frac{|Rm_k|}{f_k}(x_k)\geq k$, for all $k\geq 1$.
\\

We first consider Case \textbf{I}. We use Claim \ref{Point_Picking} to find $y_k\in B_{g_k}(x_k,1)$, $Q_k=|Rm_k|(y_k)\geq|Rm_k|(x_k)\rightarrow\infty$, and $A_k\rightarrow\infty$ such that (\ref{eq:Point_Picking}) holds. By (\ref{eq:potential}) we have
\begin{eqnarray*}
R_k+|\nabla f_k|^2=f_k\leq 10^5
\end{eqnarray*}
on $B_{g_k}(y_k,A_k Q_k^{-\frac{1}{2}})\subset B_{g_k}(p_k,202)$. We scale $g_k$ with the factor $Q_k$ and use the notations with overlines to denote the scaled geometric quantities, that is, $\overline{g}_k=Q_kg_k$, $\overline{Rm}_k=Rm(\overline{g}_k)$, etc. Then we have that
\begin{eqnarray} \label{eq:1}
\overline{Ric}_k+\overline{\nabla}^2f_k=\frac{Q_k^{-1}}{2}\overline{g}_k,
\end{eqnarray}
and that
\begin{eqnarray}
\left|\overline{Rm}_k\right|&\leq& 2, \label{eq:2}
\\
\overline{R}_k+\left|\overline{\nabla}f_k\right|^2&\leq&\frac{10^5}{Q_k}:=r_k^2\rightarrow 0, \label{eq:special}
\end{eqnarray}
on $B_{\overline{g}_k}(y_k,A_k)$, and by Proposition \ref{Shi} and Proposition \ref{Ricci_Smallness} that
\begin{eqnarray}
\left|\overline{\nabla}\,\overline{Rm}_k\right|&\leq& C_1, \label{eq:Rm_grad}
\\
\left|\overline{Ric}_k\right|&\leq& C_2 r_k, \label{eq:Ric_small}
\end{eqnarray}
on $B_{\overline{g}_k}(y_k,A_k-2)$, where $C_1$ is a constant depending only on the dimension $n=4$, and $C_2$ is a constant depending only on the dimension $n=4$ and $\kappa>0$. We can apply (\ref{eq:2}), (\ref{eq:Rm_grad}), and the strong $\kappa$-noncollapsing assumption to extract from $\{(B_{\overline{g}_k}(y_k,A_k-2),\overline{g}_k,y_k)\}_{k=1}^\infty$ a subsequence that converges in the pointed $C^{2,\alpha}$ Cheeger-Gromov sense to a complete nonflat Riemannian manifold $(M_\infty,g_\infty,y_\infty)$ with $|Rm_\infty|(x_\infty)=1$. By (\ref{eq:Ric_small}), $(M_\infty,g_\infty)$ must be Ricci-flat and therefore has Euclidean volume growth, since it is also strongly $\kappa$-noncollapsed. By Cheeger and Naber (Corollary 8.86 of \cite{cheeger2015regularity}), $(M_\infty,g_\infty)$ is asymptotically locally Euclidean (ALE). By the definition of ALE, we have that outside a compact set $M_\infty$ is diffeomorphic to a finite quotient of $\mathbb{R}^4\setminus B(O,1)$, it follows that there exists an open set $U_\infty\subset M_\infty$ containing the point $y_\infty$, such that $\overline{U}_\infty$ is compact and that $M_\infty$ is diffeomorphic to $U_\infty$. By the definition of the pointed Cheeger-Gromov convergence, $U_\infty$ can be embedded in infinitely many elements of the sequence $\{(M^4_k,f_k,g_k,p_k)\}_{k=1}^\infty$, every one of which can be embedded in a closed four-manifold with vanishing second homology group; it follows that $U_\infty$ can also be imbedded in such a closed four-manifold with vanishing second homology group, which is a contradiction against Theorem \ref{Topology}.
\\

Case \textbf{II} is almost the same as Case \textbf{I}. By the same point picking and scaling method we also get (\ref{eq:1}), (\ref{eq:2}), (\ref{eq:Rm_grad}), and (\ref{eq:Ric_small}). The only place where special care should be taken is (\ref{eq:special}). Notice that by (\ref{eq:potential}), we have that $f_k(x)\geq 1000$ whenever $dist_{g_k}(x,p_k)\geq 198$. Moreover, since $\displaystyle\left|\nabla \sqrt{f_k}\right|\leq\frac{1}{2}$, we have
\begin{eqnarray*}
\sqrt{f_k(x)}\leq\sqrt{f_k(x_k)}+1\leq\sqrt{\frac{10}{9}f_k(x_k)},
\end{eqnarray*}
for all $x\in B_{g_k}\left(y_k,A_kQ_k^{-\frac{1}{2}}\right)\subset B_{g_k}(x_k,2)$. It follows that
\begin{eqnarray*}
\overline{R}_k+\left|\overline{\nabla}f_k\right|^2=\frac{f_k}{Q_k}\leq\frac{10}{9}\frac{f_k(x_k)}{|Rm_k|(x_k)}:=r_k^2\rightarrow 0,
\end{eqnarray*}
on $B_{\overline{g}_k}(y_k,A_k)$. Therefore (\ref{eq:special}) also holds in Case \textbf{II} and we obtain the same contradiction as in Case \textbf{I}.

\end{proof}
\bigskip

\section{Proof of the main theorem}

\begin{proof} [Proof of Theorem \ref{Main_Theorem}]

\newtheorem{Claim}{Claim}

By Theorem \ref{BCBD}, Proposition \ref{Shi}, and (\ref{eq:potential}), we obtain locally uniform bounds for the curvatures, the derivatives of the curvatures, and the potential functions for any sequence in the space $\mathcal{M}^4(\kappa)$. Applying the standard regularity theorem to the elliptic equation
\begin{eqnarray*}
\Delta f =\frac{n}{2}-R,
\end{eqnarray*}
we also obtain locally uniform bounds for the derivatives of the potential functions. Hence we can extract from any sequence contained in $\mathcal{M}^4(\kappa)$ a subsequence that converges in the smooth pointed Cheeger-Gromov sense to a shrinking gradient Ricci soliton, also normalized as in (\ref{eq:normalization}). It remains to show that the limit Ricci shrinker is nonflat.
\\

Let $\{(M_k,g_k,f_k,p_k)\}_{k=1}^\infty\subset\mathcal{M}^4(\kappa)$, all normalized as in (\ref{eq:normalization}), and let $(M_\infty,g_\infty,f_\infty,p_\infty)$ be their limit Ricci shrinker in the smooth pointed Cheeger-Gromov sense, also normalized as in (\ref{eq:normalization}). Then we have that the following holds.
\\

\begin{Claim}
\begin{eqnarray}\label{eq:f_volume}
\operatorname{Vol}_f(g_\infty)=\lim_{k\rightarrow\infty}\operatorname{Vol}_f(g_k),
\end{eqnarray}
where $\operatorname{Vol}_f$ is the $f$-volume defined by $\displaystyle\operatorname{Vol}_f(g)=\int_M e^{-f}dg$.
\end{Claim}

\begin{proof}[Proof of the Claim]
By the uniform rapid decay of $e^{-f}$ (\ref{eq:potential}) and the uniform volume growth bound (\ref{eq:volume}) we have that for any $\eta>0$, there exists $A_0<\infty$ such that for all $A>A_0$ it holds that
\begin{eqnarray*}
\operatorname{Vol}_f(g_k)-\eta<\int_{B_{g_k}(p_k,A)}e^{-f_k}dg_k\leq \operatorname{Vol}_f(g_k),
\end{eqnarray*}
for every $k\geq 1$. The conclusion follows from first taking $k\rightarrow\infty$, and then $A\rightarrow\infty$, and finally $\eta\rightarrow 0$.
\end{proof}

\bigskip

By Proposition \ref{Gap} and (\ref{eq:f_volume}) we have
\begin{eqnarray*}
\operatorname{Vol}_f(g_\infty)=\lim_{k\rightarrow\infty}\operatorname{Vol}_f(g_k)\leq (4\pi)^{\frac{n}{2}}(1-\varepsilon),
\end{eqnarray*}
where $\varepsilon>0$ is given by Proposition \ref{Gap}. Hence $(M_\infty,g_\infty,f_\infty,p_\infty)$ is not flat, because the $f$-volume of the Gaussian shrinker is $(4\pi)^{\frac{n}{2}}$.

\end{proof}

\bigskip

\begin{proof}[Proof of Corollary \ref{Main_Corollary}]
To prove (a) we argue by contradiction. Suppose there exists $\{(M_k,g_k,f_k,p_k)\}_{k=1}^\infty\subset\mathcal{M}^4(\kappa)$ such that $R_k(p_k)\rightarrow 0$. By Theorem \ref{Main_Theorem} we can extract a subsequence that converges to a shrinking gradient Ricci soliton $(M_\infty,g_\infty,f_\infty,p_\infty)$ with $R_\infty(p_\infty)=0$, which by Chen \cite{chen2009strong} is flat, a contradiction.
\\

To prove (b), we recall that by the proof of Chow, Lu, and Yang \cite{chow2011lower}, we only need a uniform upper bound for $f$ and a uniform lower bound for $R$ on a sufficiently large ball, say $B(p,1000)$, where the former is given by (\ref{eq:potential}) and the latter is proved in the same way as for (a).
\\

To prove (c), we claim that there exist $c>0$, depending only on $\kappa$, such that $\displaystyle\operatorname{Vol}_f(g)>c$ for all $(M,g,f,p)\in\mathcal{M}^4(\kappa)$. Suppose this is not true. As in the proof of (a), we can find a sequence of counterexamples converging to a Ricci shrinker $(M_\infty,g_\infty,f_\infty,p_\infty)$ with $\displaystyle\operatorname{Vol}_f(g_\infty)=0$, which is a contradiction. Hence we have $\displaystyle\operatorname{Vol}_f(g)\in[c,(4\pi)^{\frac{n}{2}}(1-\varepsilon)]$ for all $(M,g,f,p)\in\mathcal{M}^4(\kappa)$. The conclusion follows from Wylie \cite{wylie2008complete} and Chow and Lu \cite{chow2016bound}.

\end{proof}

\bigskip

\section{Excluding instantons by a topological condition}

In this section we provide an alternative proof for Corollary 5.8 of \cite{anderson2010survey}. This proof is based in essence altogether on the personal notes of Richard Bamler, to whom we are indebted for graciously providing them to the author. However, any mistakes in transcription is solely due to the author. Forasmuch as Anderson's result is of fundamental importance to our main theorem, we include this section for the sake of completeness to help the readers to follow some details.
\\

\newtheorem{Topology}{Theorem}[section]
\newtheorem{Lemma_6_1}[Topology]{Lemma}
\newtheorem{Lemma_6_2}[Topology]{Lemma}
\newtheorem{Lemma_6_3}[Topology]{Lemma}
\newtheorem{Lemma_6_4}[Topology]{Lemma}

\begin{Topology} \label{Topology}
Let $N$ be a smooth closed $4$-dimensional manifold such that
\begin{eqnarray} \label{Homology_0}
H_2(N)=0,
\end{eqnarray}
where $H_2$ is the second homology group with coefficients in $\mathbb{Z}$. Then there is no open subset $U\subset N$ with the property that $U$ admits an Einstein ALE metric.
\end{Topology}

\bigskip

We split the proof into the following lemmas.
\\

\begin{Lemma_6_4}
Let $N$ be the closed manifold in the statement of Theorem \ref{Topology}. Let $U\subset N$ be an connected open subset such that $\partial U\cong \mathbb{S}^3/\Gamma$ and $H_1(U,\partial U)=0$, where $\Gamma$ is a finite group. Then the following hold.
\begin{eqnarray}
H_1(\partial U)&=&H_1(U)\oplus H_1(U), \label{H_1_split}
\\
H_2(U)&=&0. \label{betti}
\end{eqnarray}
\end{Lemma_6_4}

\begin{proof}
By Poincar\'{e} duality, we have
\begin{eqnarray*}
H^2(N;\mathbb{Z})\cong H_2(N)=0.
\end{eqnarray*}
By the universal coefficient theorem, we have
\begin{eqnarray*}
0=H^2(N;\mathbb{Z})\cong\operatorname{Hom}(H_2(N),\mathbb{Z})\oplus \operatorname{Ext}(H_1(N),\mathbb{Z}),
\end{eqnarray*}
which implies that $H_1(N)$ is torsion free, so henceforth we may assume
\begin{eqnarray} \label{Torsion_free}
H_1(N)\cong \mathbb{Z}^d,
\end{eqnarray}
where $d\geq 0$. By Poincar\'{e} duality and by the universal coefficient theorem again, we have
\begin{eqnarray*}
H_2(\partial U)\cong H^1(\partial U)\cong \operatorname{Hom}(H_1(\partial U),\mathbb{Z})=0,
\end{eqnarray*}
where the last equality is because $H_1(\partial U)$ is a finite abelian group and hence purely torsion. Then we have the following Mayer-Vietoris sequence
\begin{eqnarray*}
0= H_2(\partial U)\rightarrow H_2(U)\oplus H_2(N\backslash U)\rightarrow H_2(N)=0,
\end{eqnarray*}
from whence we obtain
\begin{eqnarray} \label{H_2(N)=0}
H_2(U)=H_2(N\backslash U)=0.
\end{eqnarray}
\\

On the other hand, we consider the long exact sequence
\begin{eqnarray} \label{Long_exact}
0=H_2(U)\rightarrow H_2(U,\partial U)\rightarrow H_1(\partial U)\rightarrow H_1(U)\rightarrow H_1(U,\partial U)=0.
\end{eqnarray}
It follows that $H_1(U)$ is purely torsion since the third homomorphism above is surjective and $H_1(\partial U)$ is finite. Hence by the universal coefficient theorem and by Poincar\'{e}-Lefschetz duality we have
\begin{eqnarray*}
H_2(U,\partial U)\cong H^2(U)\cong \operatorname{Hom}(H_2(U),\mathbb{Z})\oplus\operatorname{Ext}(H_1(U),\mathbb{Z})\cong H_1(U),
\end{eqnarray*}
where the last isomorphism follows from the facts that $H_1(U)$ is purely torsion and (\ref{H_2(N)=0}). Hence (\ref{Long_exact}) is simplified as
\begin{eqnarray} \label{Long_exact_1}
0\rightarrow H_1(U)\rightarrow H_1(\partial U)\rightarrow H_1(U)\rightarrow 0.
\end{eqnarray}
To see (\ref{Long_exact_1}) splits, we consider the following Mayer-Vietoris sequence
\begin{eqnarray*}
0=H_2(N)\rightarrow H_1(\partial U)\rightarrow H_1(U)\oplus H_1(N\backslash U)\rightarrow H_1(N)\cong \mathbb{Z}^d,
\end{eqnarray*}
where we have used (\ref{Torsion_free}). If we write $H_1(U)\oplus H_1(N\backslash U)\cong H_1(U)\oplus T\oplus \mathbb{Z}^e$, where $T$ is the torsion part of $H_1(N\backslash U)$, since $H_1(\partial U)$ is purely torsion, we have that the image of the second homomorphism in the above sequence is in $H_1(U)\oplus T\oplus\{0\}$, whose image under the third homomorphism is $0$. Hence we can simplify the above sequence as
\begin{eqnarray*}
0\rightarrow H_1(\partial U)\rightarrow H_1(U)\oplus T\rightarrow 0.
\end{eqnarray*}
The inclusion $H_1(U)\hookrightarrow H_1(U)\oplus T\cong H_1(\partial U)$ gives a homomorphism $H_1(U)\rightarrow H_1(\partial U)$, whose composition with the third homomorphism in (\ref{Long_exact_1}) is the identity on $H_1(U)$. It follows that (\ref{Long_exact_1}) splits and we have completed the proof.

\end{proof}

\bigskip

\begin{Lemma_6_1} \label{Lemma_6_1}
Let $\mathbb{S}^3/\Gamma$ be a round space form, where $\Gamma$ is a finite group. If $H_1(\mathbb{S}^3/\Gamma)\cong G\oplus G$, for some group $G$, then either $\Gamma$ is the binary dihedral group $D^*_n$ with $n$ being even, or $\Gamma$ is the binary icosahedral group with order $120$.
\end{Lemma_6_1}

\begin{proof}

We shall check every possible group $\Gamma$.
\\

\begin{enumerate}[(a)]
\item \textbf{Lens space:} In this case $H_1(\mathbb{S}^3/\Gamma)\cong\Gamma=\mathbb{Z}_m$ with $m\geq 2$, which is not possible.

\item \textbf{Prism manifold:} In this case the fundamental group has the presentation
  \begin{eqnarray*}
  \langle x,y|xyx^{-1}=y^{-1},x^{2^k}=y^n\rangle\times \mathbb{Z}_m,
  \end{eqnarray*}
  where $k$, $m\geq 1$, $n\geq 2$, and $m$ is coprime to $2n$. Its abelianization is
  \begin{eqnarray*}
  H_1(\mathbb{S}^3/\Gamma)\cong\langle x,y|y=y^{-1},x^{2^k}=y^n\rangle\times \mathbb{Z}_m,
  \end{eqnarray*}
  where $y^2=1$. We have that $H_1(\mathbb{S}^3/\Gamma)\cong \mathbb{Z}_2\times\mathbb{Z}_{2^k}\times \mathbb{Z}_m$ in the case $n$ is even, and that $H_1(\mathbb{S}^3/\Gamma)\cong \mathbb{Z}_{2^{k+1}}\times \mathbb{Z}_m$ in the case $n$ is odd. Since $m$ is coprime to $2n$, we have that the only possible case is when $m=1$, $n$ is even, and $k=1$.

\item \textbf{Tetrahedral manifold:} In this case we have
  \begin{eqnarray*}
  \Gamma=\langle x,y,z|(xy)^2=x^2=y^2,zxz^{-1}=y,zyz^{-1}=xy,z^{3^k}=1\rangle\times \mathbb{Z}_m,
  \end{eqnarray*}
  where $k$, $m\geq 1$ and $m$ is coprime to $6$. Then we have
  \begin{eqnarray*}
  H_1(\mathbb{S}^3/\Gamma)&\cong&\langle x,y,z|x^2=y^2=1,x=y,y=xy,z^{3^k}=1\rangle\times \mathbb{Z}_m
  \\
  &=&\langle x,z|x^2=1,x=x^2,z^{3^k}=1\rangle\times \mathbb{Z}_m=\mathbb{Z}_{3^k}\times\mathbb{Z}_m.
  \end{eqnarray*}
  Since $m$ is coprime to $6$, this case is not possible.

\item \textbf{Octahedral manifold:} In this case we have
  \begin{eqnarray*}
  \Gamma=\langle x,y|(xy)^2=x^3=y^4\rangle\times \mathbb{Z}_m,
  \end{eqnarray*}
  where $m$ is coprime to $6$. Then we have
  \begin{eqnarray*}
  H_1(\mathbb{S}^3/\Gamma)\cong\langle x,y|x=y^2=x^2\rangle\times \mathbb{Z}_m=\mathbb{Z}_2\times\mathbb{Z}_m.
  \end{eqnarray*}
  Since $m$ is coprime to $6$, this case is not possible.

\item \textbf{Icosahedral manifold:} In this case we have
  \begin{eqnarray*}
  \Gamma=\langle x,y|(xy)^2=x^3=x^3y^5\rangle\times \mathbb{Z}_m,
  \end{eqnarray*}
  where $m$ is coprime to $30$. Then we have
  \begin{eqnarray*}
  H_1(\mathbb{S}^3/\Gamma)&\cong&\langle x,y|x=y^2,x^2=y^3\rangle\times\mathbb{Z}_m
  \\
  &=&\langle x,y|x=y^2,x^2=y^3,y=1\rangle\times\mathbb{Z}_m=\mathbb{Z}_m.
  \end{eqnarray*}
  The only possibility is $m=1$.
\end{enumerate}

\end{proof}

\bigskip

We still need to consider the two cases when $\Gamma$ is the binary dihedral group $D^*_{2n}$ or the binary icosahedral group. In both cases $\Gamma$ can be embedded in $SU(2)$. Indeed, it is well known that the binary dihedral, tetrahedral, octahedral, and icosahedral groups are all finite subgroups of $SU(2)$; see \cite{kronheimer1989construction}.
\\

\begin{Lemma_6_2} \label{gamma_embedded}
Let $\mathbb{S}^3/\Gamma$ be the spherical space form with $\Gamma< O(4)$ being either the binary dihedral group $D^*_{2n}$ or the binary icosahedral group. Then there exists a complex structure on $\mathbb{R}^4$ such that $\Gamma<SU(2)$
\end{Lemma_6_2}

\bigskip

\begin{Lemma_6_3} \label{gauss_bonnet}
Let $(U,g)$ be an Einstein ALE space which is asymptotic to $\mathbb{S}^3/\Gamma$, where $\Gamma< SU(2)$ is isomorphic to the binary dihedral group $D^*_{2n}$ or to the binary icosahedral group. Then $b_2(U)\geq 1$.
\end{Lemma_6_3}
\begin{proof}
Assume that $b_2(U)=0$. Then we have $\chi(U)=1-b_1(U)-b_3(U)\leq 1$ and $\tau(U)=0$. Using the Chern-Gauss-Bonnet theorem and the Atiyah-Patodi-Singer index theorem, we have (see (4.4) and (4.5) in \cite{nakajima1990self})
\begin{eqnarray*}
1&\geq&\chi(U)=\frac{1}{8\pi^2}\int_U|W|^2dg+\frac{1}{|\Gamma|},
\\
0&=&\tau(U)=\frac{1}{12\pi^2}\int_U(|W^+|^2-|W^-|^2)dg-\eta_S(\mathbb{S}^3/\Gamma).
\end{eqnarray*}
Hence we have
\begin{eqnarray*}
\frac{2}{3}\geq\frac{2}{12\pi^2}\int_U|W^+|^2dg+\frac{2}{3|\Gamma|}+\eta_S(\mathbb{S}^3/\Gamma),
\end{eqnarray*}
which implies
\begin{eqnarray*}
\eta_S(\mathbb{S}^3/\Gamma)\leq\frac{2}{3}\left(1-\frac{1}{|\Gamma|}\right)<\frac{2}{3}.
\end{eqnarray*}
On the other hand, by \cite{nakajima1990self}, if $\Gamma$ is the binary dihedral group $D^*_{2n}$, we have
\begin{eqnarray*}
\eta_S(\mathbb{S}^3/\Gamma)=\frac{2(2n+2)^2-8(2n+2)+9}{6\cdot2n}=\frac{8n^2+1}{12n}>\frac{2}{3}.
\end{eqnarray*}
Similarly, if $\Gamma$ is the binary icosahedral group, then we have
\begin{eqnarray*}
\eta_S(\mathbb{S}^3/\Gamma)=\frac{361}{180}>\frac{2}{3}.
\end{eqnarray*}
In either case we yield a contradiction.

\end{proof}

\bigskip

\begin{proof} [Proof of Theorem \ref{Topology}]

\newtheorem{Claim_H1}{Claim}

Let $N$ be the manifold in Theorem \ref{Topology} and $U\subset N$ be a connected open subset that admits an Einstein ALE metric.
\\

\begin{Claim_H1}
\begin{eqnarray*}
H_1(U,\partial U)=0.
\end{eqnarray*}
\end{Claim_H1}

\begin{proof}[Proof of the claim]
Suppose the claim does not hold. We first show that the boundary $\partial \tilde{U}$ of the universal cover $\tilde{U}$ has more than one component. Since $\partial U$ is a deformation retraction of its collar neighbourhood, by excision we have
\begin{eqnarray*}
H_1(U/\partial U,\partial U/\partial U)=H_1(U,\partial U)\neq 0.
\end{eqnarray*}
Hence we have $\pi_1(U/\partial U)\neq 0$. Let $\gamma_0$ be a loop in $U/\partial U$ based at $\partial U/\partial U$ that is not null-homotopic. Lifting this loop to $U$ by the quotient map $q:U\rightarrow U/\partial U$, we obtain a curve $\gamma$ in $U$, whose ends lie in $\partial U$. By using the universal covering map $p:\tilde{U}\rightarrow U$, we can lift $\gamma$ to $\tilde{\gamma}$, a curve in $\tilde{U}$ whose ends lie in $\partial \tilde{U}$. If $\partial\tilde{U}$ is connected, since $\tilde{U}$ is simply connected, we have that $\tilde{\gamma}$ is homotopic to a curve that lies in $\partial\tilde{U}$. Composing this homotopy with $q\circ p$ we obtain a homotopy between $\gamma_0$ with a point; this is a contradiction. Hence $\partial \tilde{U}$ has more than one component.
\\

Next we observe that if $U$ admits an Einstein ALE metric, then we can lift this metric to $\tilde{U}$, which has more than one end. By Cheeger-Gromoll's splitting theorem, $\tilde{U}$ splits as the product of a line and a Ricci flat three-manifold; hence this metric is flat, a contradiction.
\end{proof}
\bigskip

We continue the proof of Theorem \ref{Topology}. By (\ref{betti}) we have that $b_2(U)=0$. On the other hand, combining (\ref{H_1_split}) and Lemma \ref{Lemma_6_1} we have that $\partial U\cong\mathbb{S}^3/\Gamma$, where $\Gamma$ is either the binary dihedral group $D^*_{2n}$ or the binary icosahedral group. It follows from Lemma \ref{gauss_bonnet} that $b_2(U)\geq 1$, and we obtain a contradiction.
\end{proof}

\bigskip

\textbf{Acknowledgement.} The author would like to thank his doctoral advisors, Professor Bennett Chow and Professor Lei Ni, for their constant help. He also owes many thanks to Professor Richard Bamler, who provided him with the personal notes appeared in section 6. The idea of this paper was inspired by a discussion between Professor Bennett Chow and the author.

\bibliographystyle{plain}
\bibliography{citation}

\noindent Department of Mathematics, University of California, San Diego, CA, 92093
\\ E-mail address: \verb"yoz020@ucsd.edu"

\end{document}